\documentclass[12pt,twoside]{article}
\usepackage{amssymb}
\usepackage[mathscr]{eucal}
\usepackage{eufrak}
\usepackage{amsmath,amsthm}
\usepackage{mathrsfs}
\usepackage{braket}
\usepackage[colorlinks=true,
   urlcolor=blue,           filecolor=green,      
   citecolor=green,      
   linkcolor=red,           bookmarks=true,
  unicode,
   plainpages=false,   ]{hyperref}
\usepackage{color}
\usepackage[latin1]{inputenc}
\font\bg=cmbx10 scaled 1200

\def\bp{\begin{proof}}
\def\ep{\end{proof}}
\def\wt{\widetilde}
\def\O{\Omega}
\def\g{\gamma}
\def\a{\alpha}
\def\n{\nabla}
\def\intl#1{\int\limits_{#1}}
\def\intll#1#2{\int\limits_{#1}^{#2}}
\def\dm{|\hskip-0.05cm|}
\def\OO{\Omega}
\def\displ{\displaystyle}

\def\VS{\vspace{6pt}\\\displ }
\def\rf#1{{\rm(\ref{#1})}}

\def\R{\Bbb R}
\def\N{\Bbb N}
\def\o{\"{o}}
\def\à{\`{a}}

\def\dy{\displaystyle}
\def\vep{\varepsilon}

\def\pa{\partial}
\def\cD{\mathcal{D}}
\def\be{\begin{equation}}
\def\ba{\begin{array}}
\def\ea{\end{array}}
\def\ee{\end{equation}}
\def\vs1{\vspace{1ex}}

\def\ov{\overline}

\def\é{\'{e}}
\font\sc=cmcsc10
\title{\bg Estimates of a possible gap related to the energy equality
for a class of non-Newtonian fluids\footnote{The results contained in this paper were presented by F.C. at the conference Fluids@PoliMi https://www.mate.polimi.it/events/Fluids-PoliMi/.
}}
\author{\sc  Francesca Crispo\thanks{Dipartimento di Matematica e Fisica, Universit\`{a} degli Studi della Campania ``L. Vanvitelli'', via Vivaldi 43, 81100 Caserta, Italy.}
 \and \sc Angelica Pia Di Feola
\thanks{ Dipartimento di Matematica e Fisica,  Universit\`{a} degli Studi della Campania ``L. Vanvitelli'', via Vivaldi 43, 81100 Caserta, Italy.
angelicapia.difeola@unicampania.it}
\and \sc Carlo Romano Grisanti \thanks{Dipartimento di Matematica, Universit\`a di Pisa, via Buonarroti 1/c, 56127 Pisa, Italy. carlo.romano.grisanti@unipi.it}}
\date{ }
\begin{document}
%\markboth{\footnotesize\rm Francesca Crispo, Carlo Romano Grisanti and   P.
%Maremonti} {\footnotesize\rm
%Some new properties of a suitable  weak solution to the Navier-Stokes equations}
\maketitle 
{\bf Abstract} - {\small The paper is concerned with the 3D-initial value problem for power-law fluids with shear dependent viscosity in a spatially periodic domain. The goal is the construction of a weak solution enjoying an energy equality.  The results hold assuming an initial data      $v_0\in J^2(\OO)$ and for $p\in \left(\frac 95,2\right)$. It is interesting to observe that the result is in complete agreement with the one known for the Navier-Stokes equations. Further, in both cases,
 the additional dissipation, which measures the possible  gap with the classical energy equality, is only expressed in terms of energy quantities.
} 

 \par\noindent
 \vskip -0.7true cm\noindent
\newcommand{\red}{\protect\bf}
\renewcommand\refname{\centerline
{\red {\normalsize \bf References}}}
\newtheorem{ass}
{\bf Assumption} 
\newtheorem{defi}
{\bf Definition} 
\newtheorem{tho}
{\bf Theorem} 
\newtheorem{rem}
{\sc Remark} 
\newtheorem{lemma}
{\bf Lemma} 
\newtheorem{coro}
{\bf Corollary} 
\newtheorem{prop}
{\bf Proposition} 
\renewcommand{\theequation}{\arabic{equation}}
\setcounter{section}{0}
\section{Introduction}\label{intro} This note concerns the 3D-initial  
value problem for power-law fluids in a spatially periodic domain:
\be\label{NS}\ba{l}v_t-\nabla \cdot((\mu+|\cD v|^2)^\frac{p-2}{2} \cD v)+v\cdot
\nabla v+\nabla\pi_v=0,\VS \nabla\cdot
v=0,\mbox{ in }(0,T)\times\OO,\VS v(0,x)=v_0(x),\mbox{ on
}\{0\}\times\OO,\ea\ee  where $\OO:=(0,L)^3$, $L\in (0, \infty)$, is a cube and we prescribe space-periodic boundary conditions
\be\label{NSb}v|_{\Gamma_j}=v|_{\Gamma_{j+3}},\ \nabla v|_{\Gamma_j}=\nabla v|_{\Gamma_{j+3}},\ \ {\pi_{v}}|_{\Gamma_j}={\pi_v}|_{\Gamma_{j+3}},\ee
with $\Gamma_j:=\partial \OO\cap \{x_j=0\}$, $\Gamma_{j+3}:=\partial \OO\cap \{x_j=L\}$, $j=1,2,3$.
In system \rf{NS} the symbol   $v$ denotes the kinetic field, $\pi_v$ is the pressure field,
 $v_t:=
\frac\partial{\partial t}v$, $\cD v:=\frac 12(\nabla v+\nabla v^T)$ the symmetric part of the gradient of $v$,  
 $v\cdot\nabla v:=
v_k\frac\partial{\partial x_k}v$, and $\mu$ is a nonnegative constant. For references, related both to the physical model and to the mathematical
theory of non-Newtonian fluids, we mainly refer to \cite{15, MNRR, Raj, show1}. \par
In this setting, we aim to construct a weak solution to the power-law system possessing the energy equality property.\\
Indeed, in the two-dimensional case for $p>1$, and in the 3-dimensional case for $p\geq \frac{11}{5}$ there exist global strong solutions for which the energy equality there holds. In the $3$D case, global weak solutions exist for $p>\frac{8}{5}$ (see \cite{Frehse}), for such solutions, as far as we are aware, only the energy inequality has been established. Hence, the aim of this investigation is to extend the range of $p$ for which there exist suitable solutions satisfying the energy equality property. Moreover, the present paper is part of a broader research that began with the study of energy equality in the context of Navier-Stokes equations undertaken by Crispo, Grisanti and Maremonti in \cite{CGM1, CGM3, CGM2}.
\par Our results concern the shear thinning case, therefore throughout the paper we always have $p<2$. Further, our main result holds for $p\in \left(\frac 95, 2\right)$. \par 
 In order to better state our result, we recall the following definitions. 
 We set $$\mathcal V:=\Set{\phi\in C^\infty_{per}(\OO), \nabla\cdot \phi=0, \int_{\OO}\phi(x)dx=0},$$
 $$J^{2}_{per}(\OO):=\mbox{completion of }\mathcal V\mbox{ in }  L^{2}(\OO),\ \ J^{1,q}_{per}(\OO):=\mbox{completion of }\mathcal V\mbox{ in }  W^{1,q}(\OO).$$ 
 \begin{defi}\label{WS}{\sl  Let $v_0\in J^2_{per}(\OO)$. A field $v:(0,\infty)\times\OO\to\R^3$ is said to be a weak solution to the problem \rf{NS}-\rf{NSb} corresponding to an initial datum $v_0$ if \begin{itemize}\item[{\rm1)}]for all $T>0$, $v\in L^\infty(0,T;J^2_{per}(\OO))\cap L^p(0,T;J^{1,p}_{per}(\OO))$,\item[{\rm 2)}]for all $T>0$, the field $v$
satisfies the  equation:
$$\displ\intll
0T\Big[(v,\varphi_\tau)-((\mu+|\cD v|^2)^\frac{p-2}{2}\cD v,\cD \varphi)+(v\cdot\nabla\varphi,v)\Big]d\tau=-(v_0,\varphi(0)),$$ for all $\varphi(t,x)\in C_0^\infty([0,T); \mathcal V)$\,,\item[{\rm 3)}]$\displ\lim_{t\to0}\,(v(t),\varphi)=(v_0,\varphi)\,,\mbox{ for all }\varphi\in \mathcal V\,.$\end{itemize}}
\end{defi} 
Using the Galerkin approximating sequence, we construct a weak solution. The main novelty of our result lies in the strong convergence (up to a suitable subsequence) of the approximating sequence in $L^q(0,T;J^{1,p}_{per}(\O))$, for all $q\in [1,p)$ and $T>0$, as well as in the almost everywhere in $(0,T)$ convergence of the $L^2$-norm of gradients\footnote{This strategy is successfully employed in \cite{CGM1, CGM3, CGM2} for the first time.}. Since strong convergence does not hold in $L^p(0,T;J^{1,p}_{per}(\O))$, where only the weak convergence is guaranteed, by lower semicontinuity of the norm, it leads to the energy inequality for our solution.
Motivated by this observation, the authors attempt to establish the energy equality for the solution using the energy equality satisfied by the approximating solutions and introducing some auxiliary functions. The outcome is a \emph{a sort of} energy equality, i.e., an energy equality involving additional quantities.
To the best of our knowledge, both this type of estimate and the strong convergence of gradients in such spaces are new in the literature.
\begin{tho}\label{mainT}
{\sl 
Let $p \in \left(\tfrac{9}{5},2\right)$, $\mu > 0$, and $v_0 \in J^2_{per}(\Omega)$.  
Let $\{v^N\}_{N \in \N}$ be the sequence in Proposition~\ref{existence}, which converges, in a suitable topology, to a weak solution $v$ of \eqref{NS}-\eqref{NSb}. Then, the set
$${\mathcal T}:=\left\{\tau\in(0,T): \|\nabla v^N(\tau)\|_p\rightarrow\|\nabla v(\tau)\|_p,\|\nabla v^N(\tau)\|_2\rightarrow \| \nabla v (\tau) \|_2\right\}$$
has full measure in $(0,T)$ and, for all $s,t \in \mathcal{T}$, with $s<t$, the solution $v$ satisfies
$$ \|v(t)\|_2^2 + 2 \int_s^t \|(\mu+|\mathcal{D} v|^2)^{\tfrac{p-2}{4}} \mathcal{D} v\|_2^2 \, d\tau + M(s,t) = \|v(s)\|_2^2, $$
with
\begin{align*}
M(s,t) :=\; &\lim_{\alpha \to \frac{\pi}{2}^-}\limsup_{N\to\infty}
\,2\int_{J_N(\alpha)} \int_\Omega (\mu+|\mathcal{D} v^N|^2)^{\tfrac{p-2}{2}} |\mathcal{D} v^N|^2 \, dx\, d\tau \\
=\; & -\lim_{\alpha \to \frac{\pi}{2}^-}\limsup_{N\to\infty}
\sum_h \Big(\|v^N(t_h(N,\alpha))\|_2^2 - \|v^N(s_h(N,\alpha))\|_2^2\Big),
\end{align*}
where, for any $\alpha \in [0,\frac{\pi}{2})$, the set $J_N(\alpha) \subset (s,t)$ has the following properties:
\begin{itemize}
\item $\|\nabla v^N(\tau)\|_2^{2\gamma} > \tan \alpha$, for any $\tau \in J_N(\alpha)$, where $\gamma = \zeta-1$ and $\zeta$ is the exponent in Lemma~\ref{zetadef};
\item $\dy \lim_{\alpha \to \frac{\pi}{2}^-} |J_N(\alpha)| = 0$, uniformly in $N$;
\item $J_N(\alpha) = \dy \bigcup_h (s_h(N,\alpha), t_h(N,\alpha))$,  
where the indices $h$ are at most countable and the intervals are mutually disjoint.
\end{itemize}}
\end{tho}
We remark that the gap expressions for our problem and the Newtonian case (see \cite{CGM1}) coincide, except that the $L^2$-norm of the gradients of solutions to the Navier-Stokes equations is replaced by the $L^p$-norm in the power-law system. Furthermore, in both cases, the additional dissipation, which quantifies the potential gap from the classical energy equality, is expressed only in terms of energy-related quantities. From a physical point of view, the energy relation would add a dissipative quantity which is not justifiable. Thus, the question arises of investigating the nature of these additional dissipation terms: they could be due to turbulence phenomena or to the weak regularity properties of the solution.
\par The plan of the paper is as follows: in Section \ref{PR}, we present some preliminary results, in particular proving the strong convergence of gradients; in Section \ref{EG}, we introduce the auxiliary weight function and provide estimates for the energy gap.

\section{\label{PR}Some preliminary results}
We start with the following known results.
\begin{lemma}\label{TINT}{\sl Let   $u\in W^{2,q}(\OO)\cap J^{1,q}_{per}(\OO) $. Then, there exists a constant $c$ independent of $u$ such that    
\be\label{INTIII}\dm u\dm_q+\dm \nabla u\dm_q\!\leq\! c\dm D^2  u\dm_q.\ee}\end{lemma}
\begin{proof} The result of the lemma is an easy adaptation to the space-periodic case of Lemma 2.7 in \cite{MS-Annali}.
\\It is well known that there exists a constant C, independent of $u$, such that
$$||\nabla u||_q \leq C (||D^2 u ||_q + ||u||_q),$$ 
and we prove that $||u||_q \leq C ||D^2u||_q$. We argue exactly as in \cite{MS-Annali}: we assume that for any $m\in \N$, there exists $u_m(x)\in W^{2,q}(\OO)\cap J^{1,q}_{per}(\OO) $ such that $||u_m||_q > m ||D^2u_m||_q$. So, we can define $v_m (x) := \frac{u_m(x)}{||u_m||_q}$ and there holds $||v_m||_q=1$ and $||D^2v_m||_q< \frac{1}{m}.$
Therefore, there exists a subsequence $\lbrace v_{m_k}(x) \rbrace$ converging weakly in $ W^{2,q}(\O)$ and strongly in $L^q(\O)$ to a function $v$ such that $||v||_q=1$ and $||D^2v||_q=0$. From this property, we deduce that $v(x) = a + b \cdot x$.  However,  since $v\in J^{1,q}_{per}(\OO)$, the periodicity condition ensures that $b=0$, and the zero-mean condition implies that $a=0$. This contradicts $||v||_q=1$.
\\\end{proof}
\begin{lemma}[Friedrichs's lemma]\label{FR}{\sl For all $\vep>0,$ there exists  $\kappa\in \N$ such that
\be\label{FRI}\dm u\dm_2\leq(1+\vep) \mbox{${\overset{\kappa}{\underset{j=1}\sum}}$}(u,a^j)+\vep\dm\nabla u\dm_q,\mbox{ for all }u\in W^{1,q}(\OO)\,,\ee
for any $q>\frac 65$, where $\{a^j\}$ is an orthonormal basis of $L^2(\OO)$}\,.\end{lemma}
\begin{proof} This result  is a generalization of the well known Friedrichs' lemma to $q \not= 2$. The proof is given in \cite{LSU}, Ch.II, Lemma 2.4. 
\end{proof}
For a sufficiently smooth $u$, we set
\be\label{Ip}
I_p(u):=\int_\OO(\mu+|\cD u|^2)^\frac{p-2}2|\n\cD u|^2\, dx\,.
\ee
In Lemma \ref{LSD} below, we collect some useful inequalities. For completeness we prove it, although similar inequalities are already known (\cite{BDR, MNRR}).
\begin{lemma}\label{LSD}{\sl Let $u\in C^2_{per}(\OO)$ with vanishing mean value, and $\mu>0$. For any $p\in (1,2)$, there exists a constant c, independent of $\mu$, such that
\begin{align}\label{SD1}
\|D^2 u\|_{p} & \leq c\, I_p(u)^\frac 12 \|(\mu+|\cD u|^2)^\frac 12\|_p^{\frac{2-p}{2}},\\
\label{SD4}
\|(\mu+|\cD u|^2)^\frac 12\|_{p}^{\frac p2} & \leq c(I_p(u)^\frac 12+\mu^\frac p4),\\
\label{SD2}
\|\n u\|_{3p} & \leq c(I_p(u)^\frac 1p+\mu^\frac 12)\,. 
\end{align}
}
\end{lemma}
\begin{proof}
Inequality \eqref{SD1} is standard and follows from  H\o lder's inequality with exponents $\frac 2p$ and $\frac{2}{2-p}$ and the pointwise inequality $|D^2 u|\leq c|\n \cD u|$:
\begin{align*}
 \|D^2 u\|_{p}^p &= \int_{\OO}(\mu+|\cD u|^2)^\frac {p(p-2)}{4}|D^2 u|^p(\mu+|\cD u|^2)^\frac {p(2-p)}{4}dx\\
&\leq c
 I_p(u)^{\frac p2} \|(\mu+|\cD u|^2)^\frac 12\|_p^{\frac{p(2-p)}{2}}.
\end{align*}
 For proving \eqref{SD4}, we first observe that, due to the periodicity of $u$ and the mean-value property, one can use Lemma \ref{TINT} to find
 \be\label{SD3}\|\cD u\|_p\leq c\|D^2 u\|_p.\ee
 Hence, employing estimate \eqref{SD1} in \eqref{SD3} and then Young's inequality with exponents $\frac 2p$ and $\frac {2}{2-p}$, we get 
 $$\|\cD u\|_p^\frac p2\leq  c\, I_p(u)^\frac p4 \|(\mu+|\cD u|^2)^\frac 12\|_p^{\frac{p(2-p)}{4}}\leq  \varepsilon   \|(\mu+|\cD u|^2)^\frac 12\|_p^{\frac{p}{2}}+c(\varepsilon)I_p(u)^\frac 12.$$
 Therefore, we have 
 $$\|(\mu+|\cD u|^2)^\frac 12\|_{p}^{\frac p2}\leq
  c\, \mu^\frac p4+\|\cD u\|_p^\frac p2\leq  c\, \mu^\frac p4+\varepsilon   \|(\mu+|\cD u|^2)^\frac 12\|_p^{\frac{p}{2}}+c(\varepsilon)I_p(u)^\frac 12,$$
  that gives estimate \eqref{SD4}.\\
Finally, let us prove inequality \eqref{SD2}. By Korn's inequality, there holds
 $$\|\n u\|_{3p}\leq c\|\cD u\|_{3p}\leq c\|(\mu+|\cD u|^2)^\frac 12\|_{3p}= c\|(\mu+|\cD u|^2)^\frac p4\|_{6}^\frac 2p.$$
 By Sobolev's embedding
 \be\label{SET}\|(\mu+|\cD u|^2)^\frac p4\|_{6}^\frac 2p\leq c(\|\nabla(\mu+|\cD u|^2)^\frac p4\|_{2}^\frac 2p+\|(\mu+|\cD u|^2)^\frac p4\|_{2}^\frac 2p).\ee
 By a direct calculation,  for the first term in \eqref{SET} we have  $$\|\nabla(\mu+|\cD u|^2)^\frac p4\|_{2}\leq c\,I_p(u),$$
while we use estimate \eqref{SD4} raised to the power of $\frac 2p$  for the second one, and we get \eqref{SD2}. \end{proof}%We consider the following mollified power-law system
%\be\label{MNSp}\ba{l}v^m_t-\nabla\cdot((\mu+|\cD v^m|^2)^\frac{p-2}{2} \cD v^m)+J_{\frac 1m}[v^m]\cdot
%\nabla v^m+\nabla\pi_{v^m}=0,\VS
%\nabla\cdot
%v^m=0,\mbox{ in }(0,T)\times\OO,\vspace{4pt}\\
%v^m=v_0^m(x)\mbox{ on
%}\{0\}\times\OO,\ea\ee 
%where \ $J_{\frac 1m}[\cdot]$ is a mollifier and $\{v_0^m\}\subset J^{1,2}_{per}(\OO)$ converges to $v_0$ in $J^2_{per}(\OO)$. \vskip0.1cm
%The result of existence is established proving that the sequence of solutions $\{v^m\}$ to problem \rf{MNSp} converges with respect to the metric stated in Definition\,\ref{WS}, to a limit function $v$ satisfying the energy inequality \rf{SEI}. 
For the existence of a weak solution $(v,\pi)$ to \rf{NS}-\rf{NSb}, we can employ the well-known Faedo-Galerkin method, as proposed in \cite{MNRR}, Chapter 5, Section 3. 
\par Let  $v^N$ be defined as 
$$v^N(t,x):=\sum_{r=1}^N c_r^N(t) a^r(x),$$ 
where the functions $a^r(x)$, in $L^2$-theory, are eigenvectors of the Stokes operator, with the corresponding eigenvalues $\lambda_r$, and the coefficients $c_r^N(t)$ are determined in such a way $v^N$ satisfies the following properties:
\be\label{ap1}\ba{ll}\displaystyle\vs1
(v^N_t,a^r)+\left( 
(\mu+|\cD v^N|^2)^\frac{p-2}2\cD v^N, \cD a^r \right)+(v^N\cdot\nabla v^N, a^r)=0\,,\  r=1,\cdots,N,\\
v^N(0)= \displaystyle \sum_{r=1}^N(v_0,a^r)a^r.\ea\ee  
\par In the following Lemma, we establish several a priori estimates for such approximating sequence.
\begin{lemma}\label{LJAM}{\sl Let $p\in (\frac 95,2)$ and let $v^N$ be  solutions to the Galerkin system \eqref{ap1}. Then there exists a constant $C$ such that%, for all $T>0$, 
\be\label{st1}
\|v^N\|_{L^\infty((0,T);J^{2}_{per}(\OO))}+\|v^N\|_{L^p((0,T);J^{1,p}_{per}(\OO))}\leq C, \ee
%\be\label{st2}\|v^N_t\|_{L^2((0,T);L^2(\OO))}\leq C,\ee and, for all $t>s\geq0$, 
\be\label{NSM-II}\dm v^N(t)\dm_2^2+2\intll st\dm(\mu+|\cD v^N|^2)^\frac{p-2}{4} \cD v^N\dm_2^2 d\tau= \dm v^N(s)\dm_2^2\leq \dm v_0\dm_2^2\,.\ee
Moreover, for all $T>0$  the sequence $\{v^N\}_{N \in \N}$ satisfies the estimate
\be\label{CL-I} \intll0{T}\dm D^2 v^N(t)\dm_p^{2\beta}dt\leq C(\|v_0\|_2)%c\left( \frac 1{ (\mu+\dm \n v^N(T)\dm_2^2)^{ \lambda-1}}+ \int_0^T(\mu+\dm \cD v^N\dm_p^p)d\tau \right) 
, \ \ \beta=\frac{p(5p-9)}{2(-p^2+8p-9)}\,,\ee }
uniformly in $N\in\N$.
\end{lemma}
\begin{proof}
\par Estimates \eqref{st1} and \eqref{NSM-II} are standard, so we omit the details.
% Let us see how to obtain the integrability of second derivatives in $L^2(0,T; L^\frac{4}{4-p}(\OO))$ and of the time derivative in $L^2((0,T);L^2(\OO))$. 
 Multiplying \eqref{ap1} by $\lambda_rc_r^N(t)$ and summing over $r$, we obtain:
\be\label{ap2}
\frac 12 \frac{d}{dt}\dm \nabla v^N\dm_2^2-((\mu+|\cD v^N|^2)^\frac{p-2}2\cD v^N
, \cD (\Delta v^N))=(v^N\cdot\nabla v^N, \Delta v^n).
\ee
We integrate by parts on both sides, and use the following identity \eqref{oo} for the nonlinear operator:
\be\label{oo}\ba{ll}\dy
 \pa_{x_s}[(\mu+|\cD v^N|^2)^\frac{p-2}2\cD v^N]\cdot\pa_{x_s}\cD v^N \\ \dy =(\mu+|\cD v^N|^2)^\frac{p-2}2|\pa_{x_s}\cD v^N|^2+(p-2) (\mu+|\cD v^N|^2)^\frac{p-4}2(\cD v^N\cdot\pa_{x_s}\cD v^N)^2
\ea\ee we find:
\be\label{ap3}
\frac 12 \frac{d}{dt}\dm \nabla v^N\dm_2^2+(p-1)I_p(v^N)
%\int_\OO(\mu+|\cD v^N|^2)^\frac{p-2}{2}
%|\nabla\cD v^N|^2\, dx
\leq \|\nabla v^N\|_3^3,
\ee
where, in the last estimate, we have taken into account the periodicity of the functions and the identity $\int_\OO v^N_j\pa^2_{jk} v^N_i\pa_k v^N_i\, dx=0$.
Now we estimate the right-hand side. By the convexity inequality for Lebesgue spaces, there hold
\be\label{c1}
\|\n v^N\|_3\leq \|\n v^N\|_{3p}^{b}\|\n v^N\|_{p}^{1-b}, \ b=\frac{3-p}{2},\ee
and
\be\label{c2}\|\n v^N\|_3\leq \|\n v^N\|_{3p}^{c}\|\n v^N\|_{2}^{1-c}, \ c=\frac{p}{3p-2}.\ee
Therefore, writing for $\alpha\in (0,1)$, $$\|\n v^N\|_3^3=\|\n v^N\|_3^{3\alpha}\|\n v^N\|_3^{3(1-\alpha)},$$ using the previous inequalities in turn, we find
\be\label{n3}\|\n v^N\|_3^3\leq \|\n v^N\|_{p}^{3\alpha (1-b)}\|\n v^N\|_{3p}^{3\alpha b}
\|\n v^N\|_{2}^{3(1-\alpha)(1-c)}\|\n v^N\|_{3p}^{3(1-\alpha)c}.\ee
Combining  \eqref{ap3} and \eqref{n3}, we find
\be\label{ap4}
\frac 12 \frac{d}{dt}\dm \nabla v^N\dm_2^2+(p-1)I_p(v^N)
%\int_\OO(\mu+|\cD v^N|^2)^\frac{p-2}{2}
%|\nabla\cD v^N|^2\, dx
\leq \|\n v^N\|_{p}^{3\alpha (1-b)}(\|\n v^N\|_{2}^2)^{\frac 32(1-\alpha)(1-c)}\|\n v^N\|_{3p}^{3\alpha b+3(1-\alpha)c}.\ee
Using Lemma \ref{LSD},  
this implies
\begin{align}
&\frac 12 \frac{d}{dt}\dm \nabla v^N\dm_2^2+(p-1)I_p(v^N)
\leq  \dy c(\mu)\|\n v^N\|_{p}^{3\alpha (1-b)}(\|\n v^N\|_{2}^2)^{\frac 32(1-\alpha)(1-c)} \nonumber\\
&\dy+\|\n v^N\|_{p}^{3\alpha (1-b)}(\|\n v^N\|_{2}^2)^{\frac 32(1-
\alpha)(1-c)}I_p(v^N)^{\frac 3p[\alpha b+(1-\alpha)c]} =: A_1+A_2,\label{app1}
\end{align}
where $c(\mu)$ is a positive constant depending on $\mu$ that tends to zero as $\mu$ goes to zero. 
Let us estimate the terms on the right-hand side. On the second one we apply Young's inequality, and we find
$$A_2\leq \varepsilon I_p(v^N)+c(\varepsilon )\|\n v^N\|_{p}^{3\alpha (1-b)\delta'}(\|\n v^N\|_{2}^2)^{\frac 32(1-
\alpha)(1-c)\delta'}$$
provided that $\delta>1$ is such that
$$\frac 3p[\alpha b+(1-\alpha)c]\delta=1.$$
Further, also requiring that the exponent of $\|\n v^N\|_p$ is equal to $p$:
$$3\alpha (1-b)\delta'=p,$$
by algebraic computations, we find
that $\alpha=\frac {p(3p-5)}{6(p-1)}$,  which is admissible since $p>\frac 53$. 
Therefore, we end up with:
$$A_2\leq \varepsilon I_p(v^N)+c(\varepsilon )\|\n v^N\|_{p}^p(\|\n v^N\|_{2}^2)^{\frac{2(3-p)}{3p-5}}.$$
As far as $A_1$ is concerned, by easy algebraic manipulations, we can increase it as follows:
\begin{align*}
A_1 & \leq  c(\mu)\left(1+\|\n v^N\|_{p}^{3\alpha (1-b)\delta'}\right)\left(1+(\|\n v^N\|_{2}^2)^{\frac 32(1- \alpha)(1-c) \delta'} \right)\\
&=c(\mu)\left(1+\|\n v^N\|_{p}^p)\left(1+(\|\n v^N\|_{2}^2\right)^{\frac{2(3-p)}{3p-5}}\right).
\end{align*}
Finally, inserting the above estimates in \eqref{app1}, we find:
\be\label{app2}\ba{ll}\dy\vs1
\frac{d}{dt}\dm \nabla v^N\dm_2^2+c I_p(v^N)
\leq c(\mu)(1+\|\n v^N\|_{p}^p)(1+\|\n v^N\|_{2}^2)^{\lambda},\ea\ee
where we set $${\lambda}:=\frac{2(3-p)}{3p-5}.$$
Dividing by  $(1+\|\n v^N\|_{2}^2)^\lambda$ and integrating in $(0,T)$, since $\lambda>1$, we arrive at:
\be\label{app3}\ba{ll} \dy \vs1\frac{1}{\lambda-1}
\frac{1}{(1+\|\n v^N(0)\|_{2}^2)^{\lambda-1}}+c\int_0^T\frac{I_p(v^N)}{(1+\|\n v^N(\tau)\|_{2}^2)^\lambda} d\tau\\ \dy \leq \frac{1}{\lambda-1}\frac{1}{(1+\|\n v^N(T)\|_{2}^2)^{\lambda-1}}+c(\mu)\int_0^T(1+\|\n v^N(\tau)\|_{p}^p)d\tau\,,\ea
\ee
which, using \eqref{st1}, gives:
\be\label{app4}
\int_0^T\frac{I_p(v^N)}{(1+\|\n v^N(\tau)\|_{2}^2)^\lambda} d\tau \dy \leq C.\ee
Estimate \eqref{app4} is the starting point to get estimate \eqref{CL-I} for the second derivatives. By applying the reverse H\o lder's inequality and algebraic manipulations, we find from
$$I_p(v^N)\geq \|\n \cD v^N\|_p^2 \|(\mu+|\cD v^N|^2)^\frac 12\|_p^{p-2}\geq c\|\n \cD v^N\|_p^2 (\mu+\|\n v^N\|_p)^{p-2},$$
and \eqref{app4}, that
\be\label{app5}
\int_0^T\frac{\|D^2 v^N\|_p^2}{(\mu+\|\n v^N\|_p)^{2-p}(1+\|\n v^N(\tau)\|_{2}^2)^\lambda} d\tau \dy \leq C.\ee
Then, performing exactly the same calculations as in \cite{MNRR}, Chapter 5 (in order to get (3.60) from (3.59)), we arrive at \eqref{CL-I}. For completeness, we replicate the computation.
\\Define $\mathcal{K}(v^N):= \frac{\|D^2 v^N\|_p^2}{(\mu+\|\n v^N\|_p)^{2-p}(1+\|\n v^N(\tau)\|_{2}^2)^\lambda}$, using H\"older's inequality and \eqref{app5}, we obtain
\be \label{I}\ba{ll} \displaystyle \int_0^T \|D^2v^N\|_p^{2\beta}\, d\tau = \int_0^T \mathcal{K}(v^N)^\beta (\mu+\|\n v^N\|_p)^{(2-p)\beta} (1+ \| \nabla v^N \|_2^2)^{\lambda \beta}\, d\tau\\
 \leq  \displaystyle \left( \int_0^T \mathcal{K}(v^N)\, d\tau\right)^\beta \left(\int_0^T (\mu+\|\n v^N\|_p)^{\frac{(2-p)\beta}{1-\beta}} (1+ \|\nabla v^N\|_2^2)^{\frac{\lambda \beta}{1-\beta}}  \, d\tau \right)^{1-\beta} \\
 \leq  C  \displaystyle \left(\int_0^T (\mu+\|\n v^N\|_p)^{\frac{(2-p)\beta}{1-\beta}} (1+ \|\nabla v^N\|_2^2)^{\frac{\lambda \beta}{1-\beta}}  \, d\tau \right)^{1-\beta}\\
 \leq C \displaystyle \left(\int_0^T (\mu+\|\n v^N\|_p)^{\frac{(2-p)\beta}{1-\beta}}  d\tau \right)^{1-\beta} + C \displaystyle \left(\int_0^T (\mu+\|\n v^N\|_p)^{\frac{(2-p)\beta}{1-\beta}} ( \|\nabla v^N\|_2^2)^{\frac{\lambda \beta}{1-\beta}} d\tau \right)^{1-\beta}\\
 =: B_1^{1-\beta} + B_2^{1-\beta}. \ea \ee 
 Using the interpolation inequality, Sobolev embedding Theorem, and Lemma \ref{TINT} we find
 $$\| \nabla v^N\|_2 \leq  \| \nabla v^N\|_p^{\frac{5p-6}{2p}}  \| \nabla v^N\|_{\frac{3p}{3-p}}^{\frac{3(2-p)}{2p}} \leq  \| \nabla v^N\|_p^{\frac{5p-6}{2p}}  \| D^2 v^N\|_{p}^{\frac{3(2-p)}{2p}} , $$
 so 
 $$B_2 \leq C \int_0^T \left( \mu + \| \nabla v^N\|_p\right)^{ \left(2-p + \frac{5p -6}{p} \lambda \right)  \frac{\beta}{1-\beta}} \| D^2 u^N \|_p^{\frac{3(2-p)}{p} \frac{\lambda \beta}{1-\beta} } \, d\tau .$$
 Now, by  H\o lder's inequality
$$B_2 \leq C\left( \int_0^T (\mu+\|\nabla  v^N \|_p)^p\, d\tau\right)^{\frac{1}{\delta}} \left( \int_0^T \|D^2 v^N \|_p^{2\beta}\, d\tau\right)^{\frac{1}{\delta'}}, $$
where $\delta$ and $\delta'$ are chosen as follows
$$\frac{1}{\delta} := \left( \frac{2-p}{p} + \frac{5p-6}{p^2} \lambda \right) \frac{\beta}{1-\beta} \; \text{ and } \; \frac{1}{\delta'} := \frac{\lambda}{1-\beta} \frac{3(2-p)}{2p}.$$
In particular, $\beta$ is chosen in such a manner that $1=\frac{1}{\delta} + \frac{1}{\delta'}$, hence we fix $ \beta := \frac{(5p-9) p}{2(-p^2 + 8p -6)}$.
\\We remark that, since $\beta$ must  be positive,  the lower bound for the exponent $p> \frac{9}{5}$ is obtained.
\\In conclusion, since $\frac{(2-p)\beta}{1-\beta} \leq p$ and \eqref{st1}, $B_1$ is finite and we get
$$\int_0^T \|D^2 v^N\|_p^{2\beta} \, d\tau \leq C + \widetilde{C} \left( \int_0^T \|D^2 v^N \|_p^{2\beta} \, d\tau\right)^{\frac{1-\beta}{\delta}},$$
so the Young's inequality ensures \eqref{CL-I}.
\end{proof}
We end this section with one more estimate that will be useful for the study concerning the energy gap.

%
%\begin{lemma}\label{L1}{\sl
%For all $v_0\in J^2_{per}(\OO)$, there exists a weak solution $v$ of \eqref{NS} as limit of a sequence $\{v^m\}$ of smooth solutions to the mollified power-law fluid system \eqref{MNSp}.}\end{lemma} 

\begin{lemma}\label{zetadef}{\sl  Let $p\in (\frac 95,2)$ and let $v^N$ be  solutions to the Galerkin system \eqref{ap1}. For any $T>0$, there exists a constant $M>0$ such that \vskip-0.4cm
$$\intl0^t\frac{1}{(1+\|\nabla v^N\|_2^2)^{\zeta}}\bigg| \frac{d}{dt}\|\nabla v^N\|_2^2\bigg| d\tau\leq M(T), \ 
\mbox{ for all } N\in \N \mbox{ and }t>0\,,$$
with  $\zeta:=\frac{3(p-1)}{3p-5}$. }\end{lemma}
\begin{proof}
Let us consider estimate \eqref{ap3}, which we reproduce here:
\be\label{ap5}
\frac 12 \frac{d}{dt}\dm \nabla v^N\dm_2^2+(p-1)I_p(v^N)
\leq \|\nabla v^N\|_3^3.
\ee
Applying the convexity inequality \eqref{c2} and estimate \eqref{SD2} in Lemma \ref{LSD}, we estimate the right-hand side as follows:
$$ \|\nabla v^N\|_3^3\leq 
 c\mu^{\frac{3}{2}c}(\|\n v^N\|_{2}^2)^{\frac 32(1-c)}%\\
\hfill \dy+(\|\n v^N\|_{2}^2)^{\frac 32(1-c)}I_p(v^N)^{\frac 3pc},
$$
whence, by applying Young's inequality to the last term with exponent $\delta=\frac{p}{3c}$, we easily find:
\begin{align}
\dy \|\nabla v^N\|_3^3 &\leq c\mu^{\frac{3}{2}c}(\|\n v^N\|_{2}^2)^{\frac 32(1-c)}+c(\varepsilon)(\|\n v^N\|_{2}^2)^\frac {3(1-c)p}{2(p-3c)}+ \varepsilon I_p(v^N) \nonumber \\
& \dy \leq c\,(1+\|\n v^N\|_{2}^2)^{\frac {3(1-c)p}{2(p-3c)}}+
\varepsilon I_p(v^N) =
 c\,(1+\|\n v^N\|_{2}^2)^{\zeta}+
\varepsilon I_p(v^N), \label{app55}
\end{align}
since $\frac {3(1-c)p}{2(p-3c)}=\frac{3(p-1)}{3p-5}=:\zeta$. 
Hence, estimate \eqref{ap5} becomes:
\be\label{app6}\ba{ll}\dy\vs1
\frac{d}{dt}\dm \nabla v^N\dm_2^2+c I_p(v^N)
&\dy \leq  \dy c\,(1+\|\n v^N\|_{2}^2)^{\zeta}.
\ea\ee
Dividing by 
$(1+\|\n v^N\|_{2}^2)^\zeta$ and integrating in $(0,T)$, since $\zeta>1$, we arrive at
$$\dy \vs1\frac{1}{\zeta-1}
\frac{1}{(1+\|\n v^N(0)\|_{2}^2)^{\zeta-1}}+c\int_0^T\frac{I_p(v^N)}{(1+\|\n v^N(\tau)\|_{2}^2)^\zeta} d\tau\dy \leq \frac{1}{\zeta-1}\frac{1}{(1+\|\n v^N(T)\|_{2}^2)^{\zeta-1}}+cT,$$ which furnishes, in particular,
\be\label{app8}\int_0^T\frac{I_p(v^N)}{(1+\|\n v^N(\tau)\|_{2}^2)^\zeta} d\tau\dy \leq c+cT.\ee
On the other hand, identity \eqref{ap2}, which we reproduce here
$$
\frac 12 \frac{d}{dt}\dm \nabla v^N\dm_2^2=((\mu+|\cD v^N|^2)^\frac{p-2}2\cD v^N
, \cD (\Delta v^N))+(v^N\cdot\nabla v^N, \Delta v^n),$$
taking into account identity \eqref{oo}, 
ensures that:
\begin{align}
\frac 12 \bigg|\frac{d}{dt}\dm \nabla v^N\dm_2^2\bigg| &=\bigg|-(\pa_{x_s}[(\mu+|\cD v^N|^2)^\frac{p-2}2\cD v^N],\pa_{x_s}\cD v^N)+(v^N\cdot\nabla v^N, \Delta v^n)\bigg| \nonumber \\
&\dy\leq  (3-p)\,I_p(v^N)+\|\nabla v^N\|_3^3\leq c I_p(v^N)+c\,(1+\|\n v^N\|_{2}^2)^{\zeta}, \label{app9}
\end{align}
where, in the last step,  we have employed estimate \eqref{app55}. 
Dividing both sides by $(1+\|\n v^N\|_{2}^2)^{\zeta}$ and integrating in $(0,t)$, we find:
$$ \frac 12\int_0^t\frac{1}{(1+\|\n v^N(\tau)\|_{2}^2)^{\zeta}} \bigg|\frac{d}{d\tau}\dm \nabla v^N(\tau)\dm_2^2\bigg|d\tau \leq  c\int_0^t  \frac{I_p(v^N(\tau))}{(1+\|\n v^N(\tau)\|_{2}^2)^{\zeta}}d\tau+c\int_0^t d\tau.$$
By estimating the right-hand side with \eqref{app8}, we obtain the thesis. 
%\be\label{app7}\ba{ll}\dy\vs1
%\frac{1}{(\,1+\|\n v^N\|_{2}^2)^\gamma}\frac{d}{dt}\dm \nabla v^N\dm_2^2+\frac{I_p(v^N)}{(\,1+\|\n v^N\|_{2}^2)^\gamma}
%&\dy \leq  \dy c\ea\ee
\end{proof}
\begin{prop}\label{existence}
{\sl Let $v_0 \in J^2_{per}(\Omega)$. Then the sequence $\{v^N\}$ of solutions to the Galerkin approximating system \eqref{ap1} converges, in a suitable topology, to a weak solution $v$ of \eqref{NS}-\eqref{NSb}}.
\end{prop}
\begin{proof}
See \cite[Chapter 5, Theorem~3.4]{MNRR}.
\end{proof}
 \subsection{The strong convergence of gradients}
\par The aim in this Section is to achieve the convergence property of the approximating sequence, using the estimates obtained in Lemma \ref{LJAM}.\\ 
We shall use Bochner-like spaces with time summability strictly less than one. Namely, if $X$ is a Banach space, for $\sigma \in (0,1)$,  we define\footnote{In analogy with Definition 1.2.15 in \cite{ABS}.}  $L^\sigma\left(0,T;X\right)$ as the linear space of all (equivalence classes of) strongly $\mu$-measurable function $u: (0,T) \mapsto X$ for which $\int_0^T\|u(t)\|_X^\sigma\,dt<+\infty.$\\
We remark that in the case $\sigma<1$, the quantity $\left(\int_0^T\|u(t)\|_X^\sigma\,dt\right)^{\frac1\sigma}$ is merely a quasi-norm, but the above space is equipped with a metric for which the following completeness result holds true:
 \begin{lemma}\label{Bochner}
Let $X$ be a Banach space and $0<\sigma<1$. If $\lbrace u^n \rbrace_{n\in \N}$ is a sequence in $L^\sigma\left(0,T;X\right)$ obeying to the following Cauchy condition
$$\lim_{m,n\to\infty}\int\limits_0^T\|u^n(t)-u^m(t)\|_X^\sigma\,dt=0,$$
then there exists a subsequence $ \lbrace u^{n_j} \rbrace_{j\in \N}$  and a function $u\in L^\sigma\left(0,T;X\right)$ such that:
$$\lim_{n\to\infty}\int\limits_0^T\|u^n(t)-u(t)\|_X^\sigma\,dt=0,\qquad \lim_{j\to\infty}\|u^{n_j}(t)\|_X=\|u(t)\|_X,\ \mbox{for a.e. } t\in[0,T].$$
\end{lemma}
\begin{proof} For the sake of completeness, we include the following proof.\\
By virtue of the Cauchy condition, we can find a strictly increasing sequence $\lbrace n_j\rbrace$ such that:
\be\label{Cauchy}\int\limits_0^T\|u^{n_j}(t)-u^n(t)\|_X^\sigma\,dt<2^{-j},\qquad\forall\,n>n_j.\ee
Applying the monotone convergence Theorem we get:
\be\label{monotone}\int\limits_0^T\sum_{j=2}^\infty\|u^{n_j}(t)-u^{n_{j-1}}(t)\|_X^\sigma\,dt=\sum_{j=2}^\infty\int\limits_0^T\|u^{n_j}(t)-u^{n_{j-1}}(t)\|_X^\sigma\,dt\le\sum_{j=2}^\infty2^{-(j-1)}<+\infty,\ee
hence the integrand on the left-hand side is finite for any $t\in[0,T]\setminus E$ with $|E|=0$.
For any fixed $t\not\in E$, the series $\sum\limits_{j=2}^\infty\|u^{n_j}(t)-u^{n_{j-1}}(t)\|_X^\sigma$ is convergent; hence $\|u^{n_j}(t)-u^{n_{j-1}}(t)\|_X^\sigma<1$, if $j$ is large enough. Since $0<\sigma<1$, we have  $\|u^{n_j}(t)-u^{n_{j-1}}(t)\|_X<\|u^{n_j}(t)-u^{n_{j-1}}(t)\|_X^\sigma$, and therefore $\sum\limits_{j=2}^\infty\|u^{n_j}(t)-u^{n_{j-1}}(t)\|_X<+\infty$. Since $X$ is a complete normed space, the series $\sum\limits_{j=2}^\infty\left(u^{n_j}(t)-u^{n_{j-1}}(t)\right)$ converges in $X$ to a function $w(t)$.
We set $u(t):=u^{n_1}(t)+w(t)$.
We have that:
$$u^{n_k}(t)=u^{n_1}(t)+\sum_{j=2}^k\left(u^{n_j}(t)-u^{n_{j-1}}(t)\right), \text{ for all }\,k\ge2,$$
and 
\be\label{pointwise}\lim_{k\to\infty}\|u(t)-u^{n_k}(t)\|_X=\lim_{k\to\infty}\left\|w(t)-\sum_{j=2}^k\left(u^{n_j}(t)-u^{n_{j-1}}(t)\right)\right\|_X=0,\quad \forall\,t\not\in E.\ee
We remark that
\be\label{L1bound}\ba{l}\displaystyle\|u^{n_k}(t)\|_X^\sigma\le\|u^{n_1}(t)\|_X^\sigma+\sum_{j=2}^k\|u^{n_j}(t)-u^{n_{j-1}}(t)\|_X^\sigma\\
\hfill\displaystyle\le\|u^{n_1}(t)\|_X^\sigma+\sum_{j=2}^\infty\|u^{n_j}(t)-u^{n_{j-1}}(t)\|_X^\sigma\ea\ee
and the function on the right-hand side belongs to $L^1(0,T)$ due to \eqref{monotone}. Now, we fix $n\in\N$ and observe that, by \eqref{pointwise},
$$\lim_{k\to\infty}\|u^{n_k}(t)-u^n(t)\|_X^\sigma=\|u(t)-u^n(t)\|_X^\sigma, \quad\text{ for all } t\not\in E.$$
Recalling the Cauchy condition for the sequence, for any $\varepsilon>0$ there exists $N(\varepsilon)$ such that:
$$\int\limits_0^T\|u^{n_k}(t)-u^n(t)\|_X^\sigma\,dt<\varepsilon,\qquad\text{ for all}\,n>N(\varepsilon), \ n_k>N(\varepsilon).$$
In virtue of the bound \eqref{L1bound}, we can apply the dominated convergence theorem to get:
$$\int\limits_0^T\|u(t)-u^n(t)\|_X^\sigma\,dt=\lim_{k\to\infty}\int\limits_0^T\|u^{n_k}(t)-u^n(t)\|_X^\sigma\,dt\le\varepsilon,\qquad\forall\,n>N(\varepsilon),$$
hence
$$\lim_{n\to\infty}\int\limits_0^T\|u(t)-u^n(t)\|_X^\sigma\,dt=0.$$
%We observe that
%$$\int_0^T\|u^{n_1}(t)\|_{1,2}^\sigma+\sum_{j=2}^{\infty}\|u^{n_j}(t)-u^{n_{j-1}}(t)\|_{1,2}^\sigma\,dt\le\int_0^T\|u^{n_1}(t)\|_{1,2}^\sigma\,dt+\sum_{j=2}^\infty2^{-j}<+\infty.$$
%%%%%%
%By \eqref{Cauchy}, the above series converges totally hence the sequence $\left(u^{n_k}(t)\right)_k$ converges strongly in $W^{1,2}(\Omega)$.
%Let be
%$$\lim_{k\to\infty}u^{n_k}(t)=u(t)\in W^{1,2}(\Omega).$$
\end{proof}
\begin{prop}\label{SCG}{\sl Let $v_0$,  $v$ and $\lbrace v^N \rbrace_{N \in \N}$ as in Proposition \ref{existence}, and let $\beta$ given by \rf{CL-I}. Then, $$v^N\to v\ \mbox{ strongly in \ } L^q(0,T;J^{1,p}_{per}(\OO))\,,\mbox{ for all }q\in[1,p)\mbox{ and }T>0\,. $$
Moreover, $\nabla v\in L^\beta(0,T;L^2(\Omega))$ with
\be\label{PC}
\int_0^T\|\nabla v(t)\|_2^\beta dt\leq C(\|v_0\|_2),
\ee
and there exists a subsequence $\lbrace v^{N_j}\rbrace_{j \in \N}$  %with %$\psi(t)^\beta \in L^1(0,T)$ 
such that  
\be\label{convaeL2}\dm \n v^{N_j}(t)\dm_2\to \|\nabla v (t)\|_2, \ \mbox{ a.e. in \ } (0,T).\ee
%))\,,\mbox{ for a suitable  } \beta\in (0,1)\mbox{ and }T>0\,.$$
}\end{prop}
\begin{proof}
We recall that for $ u\in W^{2,p}(\OO)\cap J^{1,p}_{per}(\OO)$, Lemma \ref{TINT} yields
$$\dm \n u\dm_p\leq \dm D^2 u\dm_p^\frac12\dm u\dm_p^\frac12\,.$$
Hence, raising to the power $\beta$, with $\beta=\frac{p(5p-9)}{2(-p^2+8p-9)}$, % given in \eqref{CL-I}, 
then, integrating on $(0,T)$ and applying H\"older's inequality with exponent $2$, we get:
$$\intll0T\dm \n v^k(t)-\n v^N(t)\dm_p^\beta\,dt\leq \left[\intll0T\dm D^2 v^k(t)-D^2 v^N(t)\dm_p^\beta dt\right]^\frac12\left[\intll0T\dm v^k(t)-v^N(t)\dm_p^\beta\,dt\right]^\frac12\,.$$ 
By virtue of Lemma\,\ref{LJAM}, we know there exists a constant $C(\|v_0\|_2)$ such that:
\begin{align*}
 \dy \intll0T\dm \n v^k(t)-\n v^N(t)\dm_p^\beta\,dt &\leq (2C(\|v_0\|_2))^\frac12\left[\intll0T\dm v^k(t)-v^N(t)\dm_p^\beta\,dt\right]^\frac12\\
&\leq c(T)\,(2C(\|v_0\|_2))^\frac12\left[\intll0T\dm v^k(t)-v^N(t)\dm_p^p\,dt\right]^\frac{\beta}{2p},
\end{align*}
for all $ k,N\in\N,$ where in the last step we have used H\o lder's inequality.
Since, thanks to Lemma \,\ref{FR}, the strong convergence of $\{v^N\}$ in $L^p(0,T;L^2(\OO))$ holds, hence in $L^p(0,T;L^p(\OO))$, the above inequality ensures that:
\be\label{Cauchy1}
\lim_{k,N}\intll0T\dm \n v^k(t)-\n v^N(t)\dm_p^\beta\,dt=0.
\ee
Let $q\in [1,p)$. 
By using the convexity inequality for Lebesgue spaces $\|V\|_{L^\frac{q}{\beta}(0,T)}\leq 
\|V\|_{L^\frac{p}{\beta}(0,T)}^{\theta} \|V\|_{L^1(0,T)}^{1-\theta}$,
with $V:=\|\n v^k(t)-\n v^N(t)\|_p^\beta$, we find:
\be\label{Cauchy2}
\intll0T\dm \n v^k(t)-\n v^N(t)\dm_p^q\,dt\leq \left(\intll0T\dm \n v^k(t)-\n v^N(t)\dm_p^p\right)^\frac{q\theta}{p}\left(\intll0T\dm \n v^k(t)-\n v^N(t)\dm_p^\beta\,dt\right)^{\frac{q}{\beta}(1-\theta)}
\ee
 As, from  the energy inequality, $\{\nabla v^N\}$ is bounded in 
$L^p(0,T;L^p(\OO))$, uniformly with respect to $N\in\N$, estimate \eqref{Cauchy2} gives the Cauchy condition for $\{\n v^N\}$ in $L^q(0,T;L^{p}(\OO))$, for any $q\in [1,p)$, thanks to  \eqref{Cauchy1}. Therefore $\{v^N\}$ strongly converges to a function in $L^q\left(0,T;W^{1,p}(\Omega)\right)$. 
On the other hand, as $\{v^N\}$ weakly converges to $v$ in $L^p\left(0,T;W^{1,p}(\Omega)\right)$, $v$ must coincide with the strong limit in each space $L^q\left(0,T;W^{1,p}(\Omega)\right)$. This concludes the proof of the first strong convergence in the statement. \par
As far as the second convergence is concerned, Lemma \ref{TINT} yields
\be\label{gradL2byD2LpuL2}\dm \n u\dm_2\leq c\dm D^2 u\dm_p^d\dm u\dm_2^{1-d}\,,\ee
 for any $u\in W^{2,p}(\OO)\cap J^{1,p}_{per}(\OO), \ \ \mbox{ with}\ d=\frac{2p}{7p-6}\,.$
Hence, raising to the power $\beta$, with $\beta=\frac{p(5p-9)}{2(-p^2+8p-9)}$ given in \eqref{CL-I}, then, integrating on $(0,T)$ and applying H\"older's inequality with exponents $\frac 1d$ and $\frac{1}{1-d}$, we get: 
$$\intll0T\dm \n v^k(t)-\n v^N(t)\dm_2^\beta\,dt\leq \left[\intll0T\dm D^2 v^k(t)-D^2 v^N(t)\dm_p^\beta dt\right]^d\left[\intll0T\dm v^k(t)-v^N(t)\dm_2^\beta\,dt\right]^{1-d}\,.$$ 
By virtue of Lemma\,\ref{LJAM}, we know the existence of a constant $C(\|v_0\|_2)$ such that
\begin{align*}
\dy \intll0T\dm \n v^k(t)-\n v^N(t)\dm_2^\beta\,dt & \leq (2C(\|v_0\|_2))^d\left[\intll0T\dm v^k(t)-v^N(t)\dm_2^\beta\,dt\right]^{1-d}\\
&\dy \leq c\,(2C(\|v_0\|_2))^d\left[\intll0T\dm v^k(t)-v^N(t)\dm_2^p\,dt\right]^\frac{\beta(1-d)}{p}\,, 
\end{align*} for all  $k,N\in\N$. 
Since, thanks to Lemma \ref{FR}, the strong convergence of $\{v^N\}$ in $L^p(0,T;L^2(\OO))$ holds,  the above inequality ensures that:
\be\label{Cauchy3}
\lim_{k,N}\intll0T\dm \n v^k(t)-\n v^N(t)\dm_2^\beta\,dt=0.
\ee
Applying Lemma \ref{Bochner} with $X=L^2(\Omega)$, we get that there exists $\psi\in L^\beta(0,T;L^2(\Omega))$ such that:
$$\lim_{N\to \infty}\intll0T\dm \psi(t)-\n v^N(t)\dm_2^\beta\,dt=0$$
hence
$$\lim_{N\to \infty}\intll0T\dm \psi(t)-\n v^N(t)\dm_p^\beta\,dt=0.$$
By the strong convergence in $L^q(0,T;J^{1,p}_{per}(\Omega))$, we have
$$\lim_{N\to \infty}\intll0T\dm \nabla v(t)-\n v^N(t)\dm_p^\beta\,dt=0$$ 
hence $\|\nabla v(t)-\psi(t)\|_p=0$ and $\psi(t)=\nabla v(t)$ for almost every $t\in[0,T]$.
%By the validity of the following inequality
%\be\label{Cauchy4}
%|\dm \n v^k(t)\dm_2^\beta-\|\n v^N(t)\dm_2^\beta|\leq \dm \n v^k(t)-\n v^N(t)\dm_2^\beta,
%\ee
%integrating in $(0,T)$ and using \eqref{Cauchy3}, we find
%\be\label{Cauchy5}
%\lim_{k,N}\int_0^T|\dm \n v^k(t)\dm_2^\beta-\|\n v^N(t)\dm_2^\beta|dt=0. 
%\ee
The convergence \eqref{convaeL2} is also a consequence of Lemma \ref{Bochner}. Estimate \eqref{PC} follows by estimates \eqref{NSM-II}, \eqref{CL-I} and inequality \eqref{gradL2byD2LpuL2}.
\end{proof}

\section{The energy gap \label{EG}}
As demonstrated in the previous Section, the approximating sequence does not strongly converge in $L^p(0,T;J_{per}^{1,p}(\O))$, but only weakly,  to the solution  and consequently satisfies the energy \emph{inequality}. In this context, we aim to estimate and provide an explicit expression for the gap in this inequality. To this end, in the present Section, we introduce a weight function whose properties are studied in Lemma \ref{convergenza}. Finally, we prove the main Theorem, in which we obtain two equivalent expressions for the gap.\par
Let us introduce some notation. For any $\tau\in[0,T]$, we set 
$$\rho_N(\tau)=\|\nabla v^N(\tau)\|_2^2, \quad \wt\rho_N(\tau)=\int_\O(\mu+|\cD v^N|^2)^{\frac{p-2}2}|\cD v^N|^2\,dx,$$
$$\rho(\tau)=\|\nabla v(\tau)\|_2^2, \quad
\wt\rho(\tau)=\int\limits_\O(\mu+|\cD v(\tau)|^2)^{\frac{p-2}2}|\cD v(\tau)|^2\,dx.$$
Using the above notation, the energy equality for the approximating functions $v^N$ becomes
\be\label{EEm}\frac d{d\tau}\|v^N(\tau)\|_2^2+2\wt\rho_N(\tau)=0.\ee
We define the function $P:[0,\frac\pi2)\times[0,+\infty)\longrightarrow\R$ as follows
$$P(\a,\rho)=\left\{\begin{array}{ll}1&\mbox{if }0\le\rho^\g\le\tan\a,\\
\dy\frac{\pi-2\arctan(\rho^\g)}{\pi-2\a}&\mbox{if }\rho^\g>\tan\a.
\end{array}\right.$$
Moreover, let us consider
$${\mathcal T}:=\left\{\tau\in(0,T): \|\nabla v^N(\tau)\|_p\rightarrow\|\nabla v(\tau)\|_p,\|\nabla v^N(\tau)\|_2\rightarrow \| \nabla v (\tau) \|_2\right\}.$$
We recall from Proposition \ref{SCG} that, for a suitable subsequence (not relabeled), the set $\mathcal T$ has full measure in $(0,T)$.	
Now, let us fix two instants $s,t\in{\mathcal T}$ with $s<t$. We can find a real number $\ov\a\in(0,\frac\pi2)$ such that
$$\max\left\{\|\nabla v(s)\|_2^{2\g},\|\nabla v(t)\|_2^{2\g}\right\}<\tan\ov\a,$$
and an integer $\ov m$ such that
\be\label{maxnablavm}\max\left\{\|\nabla v^N(s)\|_2^{2\g},\|\nabla v^N(t)\|_2^{2\g}\right\}<\tan\a, .\ee
for all  $N\ge\ov m$  and $\a\ge\ov\a$.\\
From now on, we focus on the interval $[s,t]$. We set, for any $N\ge\ov m$, and $\a\ge\ov\a$:
\be\label{Jmalpha}
J_N(\a)=\left\{\tau\in[s,t]:\rho_N^\g(\tau)>\tan\a\right\}.
\ee
We recall that, by Lemma \ref{LJAM}, $\rho_N$ is a continuous function, hence if \break
$\displaystyle\max_{\tau\in[s,t]}\left\{\rho_N^\g(\tau)\right\}\le\tan\a$, then $J_N(\a)$ is empty; otherwise it is a non empty open set. Therefore, we can find two sequences (eventually finite) of numbers $\{s_h(N,\a)\}$ and $\{t_h(N,\a)\}$  such that the intervals $\left(s_h(N,\a),t_h(N,\a)\right)$ are mutually disjoint and 
$$J_N(\a)=\bigcup_h\left(s_h(N,\a),t_h(N,\a)\right).$$
If no confusion arises, we will omit the dependence on $N$ and $\a$ of the intervals, simply writing $(s_h,t_h)$. Another consequence of the continuity of $\rho_N$ is that 
\be\label{rhomth}\rho_N^\g(s_h)=\rho_N^\g(t_h)=\tan\a, \quad\text{for all}\,h \in \N.\ee
Now, we set
$$E_N(\a)=(s,t)\setminus \ov{J_N(\a)}.$$
\begin{lemma}(Weight function's properties)\label{convergenza}
Let consider $\rho_N(\tau)$, $\wt\rho_N(\tau) $, $\overline{\alpha}$, and $\mathcal T$ as defined above. Then, for all $\alpha>\overline{\alpha}$,
\be \label{convP}\lim_{N\to \infty} \left(\|v^N(t)\|_2^2\,P(\a,\rho_N(t))-\|v^N(s)\|_2^2\,P(\a,\rho_N(s)) \right) =\|v(t)\|_2^2-\|v(s)\|_2^2,\ee
and
\be \label{convN}\lim_{N\to \infty}\int\limits_s^t\wt\rho_N(\tau)P(\a,\rho_N(\tau))\,d\tau=\int\limits_s^t\wt\rho(\tau)P(\a,\rho(\tau))\,d\tau.\ee
Moreover,
\be \label{convA} \lim_{\a\to\frac\pi2^-}\int\limits_s^t\wt\rho(\tau)P(\a,\rho(\tau))\,d\tau=\int\limits_s^t\wt\rho(\tau)\,d\tau.\ee
\end{lemma}
\begin{proof}
Concerning the first property,  we have \eqref{maxnablavm}, observing that $s,t\in{\mathcal T}$, hence $\|v^N(s)\|_2\rightarrow\|v(s)\|_2$,\  $\|v^N(t)\|_2\rightarrow\|v(t)\|_2$ and the continuity of $P$.\\
We start proving \eqref{convN} from the following decomposition:
\begin{align}
& \dy\int\limits_s^t\wt\rho_N(\tau)\,  P(\a,\rho_N(\tau))\,d\tau \nonumber\\
&\dy=\int\limits_s^t\left(\wt\rho_N(\tau)-\wt\rho(\tau)\right)P(\a,\rho_N(\tau))\,d\tau+\int\limits_s^t\wt\rho(\tau)P(\a,\rho_N(\tau))\,d\tau. \label{Lpgradient}
\end{align}
Regarding the last integral in \eqref{Lpgradient}, we observe that:
$$\left|\wt\rho(\tau)P(\a,\rho_N(\tau))\right|\le\wt\rho(\tau)\in L^1(0,T),$$
$$\lim_{N\to \infty} P(\a,\rho_N(\tau))=P(\a,\rho(\tau)), \text{ for all }\tau\in{\mathcal T},$$
hence, by the dominated convergence Theorem:
\be\label{rhotilde}\lim_{N\to \infty}\int\limits_s^t\wt\rho(\tau)P(\a,\rho_N(\tau))\,d\tau=\int\limits_s^t\wt\rho(\tau)P(\a,\rho(\tau))\,d\tau.\ee
To evaluate the first integral on the right-hand side of \eqref{Lpgradient}, we first prove that 
\be\label{limrhotildeN}\lim_{N\to \infty}\wt\rho_N(\tau)=\wt\rho(\tau),\qquad\text{ for all }\,\tau\in{\mathcal T}.\ee
For this purpose, we recall that (see \cite[Lemma 6.3]{DEbR})
\begin{align*}
&\dy\left|\left(\mu+|\cD v^N|^2\right)^{\frac{p-2}2}\cD v^N-\left(\mu+|\cD v|^2\right)^{\frac{p-2}2}\cD v\right|\\
&\le c\frac{|\cD v^N-\cD v|}{\left(\mu+|\cD v^N|+|\cD v|\right)^{2-p}}
\le c|\cD v^N-\cD v|\,|\cD v|^{p-2},
\end{align*}
hence, applying H\"{o}lder's inequality too, we find:
 \begin{align*}
& \left|\wt\rho_N(\tau)-\wt\rho(\tau)\right| \\
&=\left|\int\limits_\O\left(\mu+|\cD v^N|^2\right)^{\frac{p-2}2}\cD v^N\left(\cD v^N-\cD v\right)\,dx\right.\\
&+\left.\int\limits_\O\left(\left(\mu+|\cD v^N|^2\right)^{\frac{p-2}2}\cD v^N-\left(\mu+|\cD v|^2\right)^{\frac{p-2}2}\cD v\right)\cD v\,dx\right|\\
&\le\int\limits_\O|\cD v^N-\cD v|\,|\cD v^N|^{p-1}\,dx+c\int\limits_\O|\cD v^N-\cD v|\,|\cD v|^{p-1}\,dx\\
&\le\|\cD v^N(\tau)-\cD v(\tau)\|_p\left(\|\cD v^N(\tau)\|_p^{p-1}+c\|\cD v(\tau)\|_p^{p-1}\right).
 \end{align*}
Due to the strong convergence of $\nabla v^N(\tau)$ to $\nabla v(\tau)$ in $L^p((0,T))$ for any $\tau\in{\mathcal T}$, the claim is proven.\\
Now, returning to the first integral on the right-hand side of \eqref{Lpgradient}, for any fixed $\eta\in(\a,\frac\pi2)$, we have:
\begin{align}
&\int\limits_s^t\left(\wt\rho_N(\tau)-\wt\rho(\tau)\right)  P(\a,\rho_N(\tau)) \,d\tau \nonumber \\
\dy &=\int\limits_s^t\chi_{E_N(\eta)}(\tau)\left(\wt\rho_N(\tau)-\wt\rho(\tau)\right)P(\a,\rho_N(\tau))\,d\tau \nonumber\\
\dy& \; \; \; +\int\limits_s^t\chi_{J_N(\eta)}(\tau)\left(\wt\rho_N(\tau)-\wt\rho(\tau)\right)P(\a,\rho_N(\tau))\,d\tau. \label{rhom-rho}
\end{align}
If $\tau\in E_N(\eta)$, then 
$$\wt\rho_N(\tau)\le\|\cD v^N(\tau)\|_p^p\le c\rho_N(\tau)^{\frac p2}\le(\tan\eta)^{\frac p{2\g}}.$$
Hence
$$\left|\chi_{E_N(\eta)}(\tau)\left(\wt\rho_N(\tau)-\wt\rho(\tau)\right)P(\a,\rho_N(\tau))\right|\le (\tan\eta)^{\frac p{2\g}}+\wt\rho(\tau)\in L^1((s,t)),$$
and, recalling \eqref{limrhotildeN} and using the dominated convergence theorem, we find
\be\label{Embeta}\lim_{N\to \infty} \int\limits_s^t\chi_{E_N(\eta)}(\tau)\left(\wt\rho_N(\tau)-\wt\rho(\tau)\right)P(\a,\rho_N(\tau))\,d\tau=0,\text{ for all } \a<\eta<\frac\pi2.\ee
If $\tau\in J_N(\eta)$, since $P(\a,\cdot)$ is a decreasing function, then, the energy identity \eqref{EEm} implies:
\begin{align}
&\dy\int\limits_s^t \left|\chi_{J_N(\eta)}(\tau)\left(\wt\rho_N(\tau)-\wt\rho(\tau)\right)P(\a,\rho_N(\tau))\right|\,  d\tau \nonumber\\
& \dy\le P(\a,(\tan\eta)^{\frac1\g})\int\limits_s^t\left(\wt\rho_N(\tau)+\wt\rho(\tau)\right)\, d\tau\le cP(\a,(\tan\eta)^{\frac1\g}),\label{Jmbeta}
\end{align}
where we have used \eqref{st1}.
Using \eqref{Embeta} and \eqref{Jmbeta} in \eqref{rhom-rho} we get
$$0\le\limsup_{N\to \infty}\int\limits_s^t\left|\wt\rho_N(\tau)-\wt\rho(\tau)\right|P(\a,\rho_N(\tau))\,d\tau\le cP(\a,(\tan\eta)^{\frac1\g}).$$
Passing to the limit as $\eta\to\frac\pi2^-$, since $\lim\limits_{\rho\to+\infty}P(\a,\rho)=0$, we have that
$$\limsup_{N \to \infty} \left|\int\limits_s^t\left(\wt\rho_N(\tau)-\wt\rho(\tau)\right)P(\a,\rho_N(\tau))\,d\tau\right|=0,\quad\text{ for all }\,\a\ge\ov\a.$$
Using this result, together with \eqref{rhotilde}, in \eqref{Lpgradient}, we get \eqref{convN}.\\
To prove \eqref{convA}, we observe that
$$0\le\wt\rho(\tau)P(\a,\rho(\tau))\le\wt\rho(\tau)\in L^1\left((s,t)\right),$$
hence, since $\lim\limits_{\a\to\frac\pi2^-}P(\a,\rho(\tau))=1$ for any $\tau\in(s,t)$, by the dominated convergence Theorem, the claim is proven.
\end{proof}
We are now ready to prove the main Theorem.
\begin{proof}[Proof of Theorem \ref{mainT}]
For reader's convenience, we rewrite the energy equality \ref{EEm} for the approximating functions $v^N$ 
$$\frac d{d\tau}\|v^N(\tau)\|_2^2+2\wt\rho_N(\tau)=0.$$
Moreover, since $E_N(\a)$ is open and $P(\a,\rho_N(\tau))=1$ for any $\tau\in E_N(\a)$, we get that 
$$\frac d{d\tau}P(\a,\rho_N(\tau))=0,\ \text{ for all }\,\tau\in E_N(\a).$$
On the other side, if $\tau\in J_N(\a)$ we have
$$\frac d{d\tau}P(\a,\rho_N(\tau))=\frac{-2}{\pi-2\a}\frac1{1+\rho_N^{2\g}(\tau)}\frac d{d\tau}(\rho_N^\g)(\tau).$$
Note that $(s,t)\setminus\left(E_N(\a)\cup J_N(\a)\right)$ is a negligible set.\\
We consider the energy identity \eqref{EEm} weighted with $P(\a,\rho_N)$, namely
$$\frac d{d\tau}\|v^N(\tau)\|_2^2\,P(\a,\rho_N(\tau))+2\wt\rho_N(\tau) \,P(\a,\rho_N(\tau))=0.$$
We integrate by parts on the interval $(s,t)$, obtaining
\begin{align}
&\|v^N(t)  \|_2^2\,P(\a,\rho_N(t)) -  \|v^N(s)\|_2^2\,  P(\a,\rho_N(s)) \nonumber\\
&\dy+\frac{2}{\pi-2\a}\int\limits_{J_N(\a)}\|v^N(\tau)\|_2^2\frac1{1+\rho_N^{2\g}(\tau)}\frac d{d\tau}(\rho_N^\g)(\tau)\,d\tau +2\int\limits_s^t\wt\rho_N(\tau)\,P(\a,\rho_N(\tau))\,d\tau=0.\label{WEEm}
\end{align}
Passing to the limit as $N\to\infty$ in \eqref{WEEm}, using \eqref{convP} and  \eqref{convN}, we get
\begin{align}
& \dy\frac{2}{\pi-2\a}\lim_{N\to \infty}\int\limits_{J_N(\a)}\frac{\|v^N(\tau)\|_2^2}{1+\rho_N^{2\g}(\tau)}  \frac d{d\tau}(\rho_N^\g)(\tau)\,  d\tau \nonumber \\
&\dy=\|v(s)\|_2^2-\|v(t)\|_2^2-2\int\limits_s^t\wt\rho(\tau)P(\a,\rho(\tau))\,d\tau.\label{firstlimitm}
\end{align}
Let us integrate by parts the integral on the left-hand side, recalling that $J_N(\a)= \displaystyle \bigcup_h\left(s_h(N,\a),t_h(N,\a)\right)$
\begin{align}&\int\limits_{J_N(\a)}\frac{\|v^N(\tau)\|_2^2}{1+\rho_N^{2\g}(\tau)}  \frac d{d\tau}(\rho_N^\g)(\tau)  \,d\tau \nonumber \\
&\dy=\sum_h\left(\frac{\|v^N(t_h)\|_2^2}{1+\rho_N^{2\g}(t_h)}\rho_N^\g(t_h)-\frac{\|v^N(s_h)\|_2^2}{1+\rho_N^{2\g}(s_h)}\rho_N^\g(s_h)\right)   - \!\!\int\limits_{J_N(\a)}\frac d{d\tau}\left(\frac{\|v^N(\tau)\|_2^2}{1+\rho_N^{2\g}(\tau)}\right)\rho_N^\g(\tau)\,d\tau. 
\label{Mstm}
\end{align}
Recalling \eqref{rhomth}, the sum on the right-hand side can be rewritten as
$$\frac{\tan\a}{1+\tan^2\a}\sum_h\left(\|v^N(t_h)\|_2^2-\|v^N(s_h)\|_2^2\right).$$
Concerning the integral on the right-hand side of \eqref{Mstm}, we compute the derivative and we use the energy identity \eqref{EEm} to get
\be\label{derivativeexpanded}\ba{l}\dy\int\limits_{J_N(\a)}\frac d{d\tau}\left(\frac{\|v^N(\tau)\|_2^2}{1+\rho_N^{2\g}(\tau)}\right)\rho_N^\g(\tau)\,d\tau
=-2\int\limits_{J_N(\a)}\wt\rho_N(\tau)\frac{\rho_N^\g(\tau)}{1+\rho_N^{2\g}(\tau)}\,d\tau\\
\dy-2\int\limits_{J_N(\a)}\|v^N(\tau)\|_2^2\frac{\rho_N^{2\g}(\tau)}{\left(1+\rho_N^{2\g}(\tau)\right)^2}\frac d{d\tau}(\rho_N^\g)(\tau)\,d\tau.
\ea\ee
Substituting the above results in \eqref{Mstm}, we get
\begin{align*}
& \dy\int\limits_{J_N(\a)} \frac{\|v^N(\tau)\|_2^2}{1+\rho_N^{2\g}(\tau)}\frac d{d\tau}(\rho_N^\g)(\tau)\,d\tau
-2\int\limits_{J_N(\a)}\frac{\|v^N(\tau)\|_2^2\,\rho_N^{2\g}(\tau)}{\left(1+\rho_N^{2\g}(\tau)\right)^2}\frac d{d\tau}(\rho_N^\g)(\tau)\,d\tau\\
&\dy=\frac{\tan\a}{1+\tan^2\a}\sum_h\left(\|v^N(t_h)\|_2^2-\|v^N(s_h)\|_2^2\right)
+2\int\limits_{J_N(\a)}\wt\rho_N(\tau)\frac{\rho_N^\g(\tau)}{1+\rho_N^{2\g}(\tau)}\,d\tau.
\end{align*}
Before proceeding, we rewrite the left-hand side of the above equality using the algebraic identity
$$\frac1{1+\rho_N^{2\g}(\tau)}-\frac{2\rho_N^{2\g}}{\left(1+\rho_N^{2\g}(\tau)\right)^2}=-\frac1{1+\rho_N^{2\g}(\tau)}+\frac2{\left(1+\rho_N^{2\g}(\tau)\right)^2},$$
obtaining 
\begin{align}
&\dy\int\limits_{J_N(\a)} \frac{\|v^N(\tau)\|_2^2}{1+\rho_N^{2\g}(\tau)}\frac d{d\tau}(\rho_N^\g)(\tau)\,d\tau =2\int\limits_{J_N(\a)}\frac{\|v^N(\tau)\|_2^2}{\left(1+\rho_N^{2\g}(\tau)\right)^2}\frac d{d\tau}\left(\rho_N^\g(\tau)\right)\,d\tau \nonumber\\
&\dy-\frac{\tan\a}{1+\tan^2\a}\sum_h\left(\|v^N(t_h)\|_2^2-\|v^N(s_h)\|_2^2\right)
-2\int\limits_{J_N(\a)}\wt\rho_N(\tau)\frac{\rho_N^\g(\tau)}{1+\rho_N^{2\g}(\tau)}\,d\tau. \label{algebraic}
\end{align}
We are going to pass to the limit as $N$ goes to infinity.\hfill\break
For the last integral, we recall that $\wt\rho_N(\tau)\le c\rho_N(\tau)^{\frac p2}$ and that $\frac p2+\g<2\g$, since $\g>1$, hence
$$0\le\wt\rho_N(\tau)\frac{\rho_N^\g(\tau)}{1+\rho_N^{2\g}(\tau)}\le c\frac{\rho_N^{\frac p2+\g}(\tau)}{1+\rho_N^{2\g}(\tau)}\le c.$$
Applying the Fatou's Lemma, we get
\begin{align*}
\dy0 &\le\limsup_{N\to \infty}\int\limits_{J_N(\a)}\wt\rho_N(\tau)\frac{\rho_N^\g(\tau)}{1+\rho_N^{2\g}(\tau)}\,d\tau\\
\dy & \le\int\limits_s^t\limsup_{N\to \infty}\left(\chi_{J_N(\a)}\wt\rho_N(\tau)\frac{\rho_N^\g(\tau)}{1+\rho_N^{2\g}(\tau)}\right)\,d\tau.
\end{align*}
Now, we set
$$J(\a)=\limsup_{N\to \infty} J_N(\a)=\bigcap_{j=0}^\infty\bigcup_{N=j}^\infty J_N(\a),$$
and we remark that 
$$\tau\in J(\a) \iff \exists\, N_k\to\infty:  \tau\in J_{N_k}(\a),\, \text{for all }\,k \in \N.$$
Therefore $\chi_{J_{N_k}(\a)}(\tau)=1$ for any $k$, which ensures that $\limsup\limits_N\chi_{J_N(\a)}(\tau)=\chi_{J(\a)}(\tau)$.
Moreover
$$\tau\in J(\a)\cap{\mathcal T}\Rightarrow\rho_{N_k}^\g(\tau)>\tan\a\Rightarrow\rho^\g(\tau)\ge\tan\a.$$
It follows that
\begin{align}
\dy0 & \le\limsup_{N\to \infty}\int\limits_{J_N(\a)}\wt\rho_N(\tau)\frac{\rho_N^\g(\tau)}{1+\rho_N^{2\g}(\tau)}\,d\tau
=\limsup_{N\to \infty}\int\limits_s^t\chi_{J_N(\a)}\wt\rho_N(\tau)\frac{\rho_N^\g(\tau)}{1+\rho_N^{2\g}(\tau)}\,d\tau \nonumber \\
&\dy\le\int\limits_s^t\limsup_{N\to \infty}\chi_{J_N(\a)}\wt\rho_N(\tau)\frac{\rho_N^\g(\tau)}{1+\rho_N^{2\g}(\tau)}\,d\tau
=\int\limits_{J(\a)}\wt\rho(\tau)\frac{\rho^\g(\tau)}{1+\rho^{2\g}(\tau)}\,d\tau \nonumber\\
&\dy\le\frac1{\tan\a}\int\limits_{J(\a)}\wt\rho(\tau)\frac{\rho^{2\g}(\tau)}{1+\rho^{2\g}(\tau)}\,d\tau\le\frac1{\tan\a}\int\limits_{J(\a)}\wt\rho(\tau)\,d\tau.\label{limsuptilderhom}
\end{align}
Going back to \eqref{algebraic}, we estimate the first integral on the right-hand side recalling the energy equality \eqref{EEm} and the definition \eqref{Jmalpha} of $J_N(\a)$ 
\begin{align*}
\dy\left|\,\int\limits_{J_N(\a)}\frac{\|v^N(\tau)\|_2^2\frac d{d\tau}(\rho_N^\g)(\tau)}{\left(1+\rho_N^{2\g}(\tau)\right)^2}\,d\tau\right|
&\le\sup_{\tau,N}\|v^N(\tau)\|_2^2\left|\,\int\limits_{J_N(\a)}\frac{\frac d{d\tau}(\rho_N^\g)(\tau)}{\left(1+\rho_N^{2\g}(\tau)\right)^2}\,d\tau\right|\\
\dy &\le c\|v_0\|_2^2\left|\,\int\limits_{J_N(\a)}\frac{\g\rho_N^{\g-1}(\tau)\frac{d}{d\tau} \rho_N(\tau)}{\left(1+\rho_N^{2\g}(\tau)\right)^2}\,d\tau\right|\\
\dy & \le\frac{c\g\|v_0\|_2^2}{1+\tan^2\a}\int\limits_{J_N(\a)}\frac{(1+\rho_N(\tau))^{\g-1}\left|\frac{d} {d\tau} \rho_N(\tau)\right|}{1+\rho_N^{2\g}(\tau)}\,d\tau\\
\dy & \le\frac{c\g\|v_0\|_2^2}{1+\tan^2\a}\int\limits_{J_N(\a)}\frac{\left|\frac d{d\tau}\rho_N(\tau)\right|}{(1+\rho_N(\tau))^{\g+1}}\,d\tau
\le\frac{c\g\|v_0\|_2^2M}{1+\tan^2\a},
\end{align*}
where the last estimate follows from Lemma \ref{zetadef} recalling that $\zeta=\g+1$. Hence
\be\label{limsupfrac1+rhom^2}
\limsup_{N\to \infty}\left|\,\int\limits_{J_N(\a)}\frac{\|v^N(\tau)\|_2^2\frac d{d\tau}(\rho_N^\g)(\tau)}{\left(1+\rho_N^{2\g}(\tau)\right)^2}\,d\tau\right|
\le\frac{c\g\|v_0\|_2^2M}{1+\tan^2\a},\qquad\forall\,\a>\ov\a.
\ee
Keeping in mind that our target is \eqref{firstlimitm}, we multiply \eqref{algebraic} by $\frac{2}{\pi-2\a}$ and we pass to the limit as $N\to\infty$. We remark that the only term for which the existence of the limit is guaranteed is the first integral on the left-hand side (due to \eqref{firstlimitm}), hence we will rather consider the $\limsup$  
\begin{align}
& \frac{2}{\pi-2\a} \lim_{N\to \infty}  \int\limits_{J_N(\a)}
\frac{\|v^N(\tau)\|_2^2}{1+\rho_N^{2\g}(\tau)}
\frac{d}{d\tau}(\rho_N^\g)(\tau)\,d\tau \nonumber\\
& + \frac{2\tan\a}{(1+\tan^2\a)(\pi-2\a)}
\limsup_{N\to \infty} \sum_h \left(\|v^N(t_h)\|_2^2-\|v^N(s_h)\|_2^2\right)= \nonumber \\
\dy &= \frac4{\pi-2\a} \limsup_{N\to \infty} \left( \int\limits_{J_N(\a)}\frac{\|v^N(\tau)\|_2^2}{\left(1+\rho_N^{2\g}(\tau)\right)^2}\frac d{d\tau}\left(\rho_N^\g(\tau)\right)\,d\tau\right.
\left. -\int\limits_{J_N(\a)}\wt\rho_N(\tau)\frac{\rho_N^\g(\tau)}{1+\rho_N^{2\g}(\tau)}\,d\tau\right) \nonumber \\
\label{limitonalpha}
& =:\frac4{\pi-2\a}\limsup_{N\to \infty}\,(C_N(\a)-D_N(\a)).
\end{align}
Concerning the right-hand side, by \eqref{limsupfrac1+rhom^2} and \eqref{limsuptilderhom}  we have
\begin{align}
&\frac4{\pi- 2\a} \left|\limsup_{N\to \infty}(C_N(\a)- D_N(\a))\right|  \nonumber \\
\dy & \le\frac4{\pi-2\a}\,\left(\limsup_{N\to \infty}|C_N(\a)|+\limsup_{N\to \infty}|D_N(\a)|\right) \nonumber \\
\dy & \le\frac4{\pi-2\a}\left(\frac{c\g\|v_0\|_2^2M}{1+\tan^2\a}+\frac1{\tan\a}\int\limits_{J(\a)}\wt\rho(\tau)\,d\tau\right) \!\!. \label{lismupCmDm}
\end{align}
We observe that 
\be\label{limittan}\lim_{\a\to\frac\pi2^-}\frac1{(\pi-2\a)(1+\tan^2\a)}=0,\qquad\lim_{\a\to\frac\pi2^-}\frac1{(\pi-2\a)\tan\a}=\frac12 , \ee
hence we need an estimate of the measure of $J(\a)$.
We recall that, if $\tau\in J(\a)$ then $\left(\frac{\rho(\tau)}{(\tan\a)^{\frac1\g}}\right)^{\frac\beta2}\ge1$, where $\beta$ is the exponent in Proposition \ref{SCG}, hence
\begin{align}
\dy|J(\a)| & \le\frac1{(\tan\a)^{\frac\beta{2\g}}}\int\limits_{J(\a)}\rho^{\frac\beta2}(\tau)\,d\tau  \nonumber \\
&\dy\le\frac1{(\tan\a)^{\frac\beta{2\g}}}\int\limits_s^t\left\|\nabla v(\tau)\right\|_2^\beta\,d\tau\le\frac{C(\|v_0\|_2)}{(\tan\a)^{\frac \beta{2\g}}}, \label{measureJalpha}
\end{align}
thanks to \eqref{PC}. Hence $\lim\limits_{\a\to\frac\pi2^-}|J(\a)|=0$, and, by absolute continuity of the Lebesgue integral, 
\be\label{limitintrhotilde}\lim_{\a\to\frac\pi2^-}\int\limits_{J(\a)}\wt\rho(\tau)\,d\tau=0.\ee
By \eqref{lismupCmDm}, \eqref{limittan}, and \eqref{limitintrhotilde} we get
\be\label{limCNDN}\lim_{\a\to\frac\pi2^-}\frac4{\pi-2\a}\limsup_{N\to \infty}(C_N(\a)-D_N(\a))=0.\ee
Passing to the limit on $\a$ in equation \eqref{limitonalpha} and using \eqref{limCNDN}, we have
\begin{align}
\dy\lim_{\a\to\frac\pi2^-} & \left(  \frac{2}{\pi-2\a}\lim_{N\to \infty}\int\limits_{J_N(\a)}\frac{\|v^N(\tau)\|_2^2}{1+\rho_N^{2\g}(\tau)}\frac d{d\tau}(\rho_N^\g)(\tau)\,d\tau\right. \nonumber \\
& \dy\left.+\frac{2\tan\a}{(1+\tan^2\a)(\pi-2\a)}\limsup_{N\to \infty}\sum_h\left(\|v^N(t_h)\|_2^2-\|v^N(s_h)\|_2^2\right)\right)=0.\label{limlimsup}
\end{align}
Using \eqref{convA}, it follows that:
\be\label{limitalphaN}\lim_{\a\to\frac\pi2^-}\frac{2}{\pi-2\a}\lim_{N\to \infty}\int\limits_{J_N(\a)}\frac{\|v^N(\tau)\|_2^2}{1+\rho_N^{2\g}(\tau)}\frac d{d\tau}(\rho_N^\g)(\tau)\,d\tau
=\|v(s)\|_2^2-\|v(t)\|_2^2-2\int\limits_s^t\wt\rho(\tau)\,d\tau.\ee
Hence the limit as $\a\to\frac\pi2^-$ of the first term in \eqref{limlimsup} exists and it is finite. Observing that 
$$\lim_{\a\to\frac\pi2^-}\frac{2\tan\a}{(1+\tan^2\a)(\pi-2\a)}=1,$$
we get, from \eqref{limlimsup} and \eqref{limitalphaN}, that
$$\lim_{\a\to\frac\pi2^-}\limsup_{N\to \infty}\sum_h\left(\|v^N(t_h)\|_2^2-\|v^N(s_h)\|_2^2\right)=\|v(t)\|_2^2-\|v(s)\|_2^2+2\int\limits_s^t\wt\rho(\tau)\,d\tau.$$
To get the second expression of the energy gap we only need to remark that, by \eqref{EEm}, we have
$$\sum_h\left(\|v^N(t_h)\|_2^2-\|v^N(s_h)\|_2^2\right)=-2\sum_h\int\limits_{s_h}^{t_h}\wt\rho_N(\tau)\,d\tau=-2\int\limits_{J_N(\a)}\wt\rho_N(\tau)\,d\tau.$$
Finally, the claim on the measure of $J_N(\a)$ follows by the estimate \eqref{measureJalpha} with $J_N(\a)$ in place of $J(\a)$ and $\nabla v^N$ in place of $\nabla v$.
\end{proof}
%Thus for all $T>0$  we deduce that  
%\be\label{Clim}\ba{ll}\displ\intll st\dm \n v(\tau)\dm_2^2d\tau\hskip-0.2cm&\displ=\lim_{p\to2^-}\intll st\dm\n v(\tau)\dm_2^pd\tau\\&\displ=\lim_{p\to2^-}\lim_{m\to\infty}\intll st\dm \n v^m(\tau)\dm_2^pd\tau\,,\mbox{ for all }t,s\in(0,T) \,.\ea\ee 
\vskip0.1cm\noindent
 {\bf Acknowledgment} - The research activity of F. Crispo and A.P. Di Feola is performed under the auspices of GNFM-INdAM. \\
 The research activity of F. Crispo is partially supported by Universit\`{a} degli Studi della Campania ``Luigi Vanvitelli'', D.R. 111/2024, within VISCOMATH project.\\
The research activity of  C.R. Grisanti is performed under the auspices of GNAMPA-INdAM. The author acknowledges the MIUR Excellence Department Project awarded to the Department of Mathematics, University of Pisa, CUP I57G22000700001.
\vskip0.1cm\noindent
 {\bf Declarations}
\vskip0.1cm\noindent
 {\bf Funding} -  No funds, grants, or other support was received.
\vskip0.1cm\noindent
 {\bf Conflict of interest} - The authors have no conflicts of interest to declare that are relevant to the content of this article.

\end{document}